\documentclass[11pt]{amsart}

\usepackage{amsmath,amssymb,amsthm}
\usepackage{booktabs}
\usepackage{array}
\usepackage[hidelinks]{hyperref}

\theoremstyle{plain}
\newtheorem{theorem}{Theorem}[section]
\newtheorem{proposition}[theorem]{Proposition}
\newtheorem{lemma}[theorem]{Lemma}

\theoremstyle{definition}
\newtheorem{definition}[theorem]{Definition}

\theoremstyle{remark}
\newtheorem{remark}[theorem]{Remark}

\makeatletter
\def\@settitle{\begin{center}%
  \baselineskip14\p@\relax
  \bfseries\Large
  \@title
  \end{center}%
}
\def\@setauthors{%
  \begingroup
  \def\thanks{\protect\thanks@warning}%
  \trivlist
  \centering
  \normalsize
  \@topsep30\p@\relax
  \advance\@topsep by -\baselineskip
  \item\relax
  \andify\authors
  \authors
  \endtrivlist
  \endgroup
}
\makeatother

\newcommand{\NN}{\mathbb{N}}
\newcommand{\ZZ}{\mathbb{Z}}
\newcommand{\RR}{\mathbb{R}}
\newcommand{\QQ}{\mathbb{Q}}
\newcommand{\dy}{\ZZ[\tfrac12]}
\newcommand{\TU}{\textup{(TU)}}

\title[Condition (TU) is false]{Isaac's Condition (TU) is false}
\author{Ralf Stephan}
\thanks{Institute for Globally Distributed Open Research and Education (IGDORE) \\
MSC 11K16, 42A38, 28A80
}
\hypersetup{
  pdftitle={Isaac's Condition (TU) is false},
  pdfauthor={Ralf Stephan}
}

\date{\today}

\begin{document}
\begin{abstract}
In 2006 Isaac reduced the simple normality to base $2$ of $\sqrt s$, for $s$
not a perfect square, to a property of dyadic expansions that he called
Condition~(TU): the upper frequency of the digit~$1$ in $\sqrt\nu$, along any
subsequence, depends only on the tail of the expansion of $\nu$. We show that
Condition~(TU) is false, already for the full sequence. We give two proofs.
The first is elementary and turns Isaac's own reduction against the condition.
The second uses the theorem of Kaufman and Mosquera--Shmerkin on the Fourier
decay of smooth images of self-similar measures, and gives counterexamples
for almost every point of a Cantor measure. All results are machine-verified
in Lean~4; the second proof uses the two cited theorems as hypotheses.
\end{abstract}
\maketitle

\section{Introduction}

A real number is \emph{simply normal to base $2$} if the digits $0$ and $1$
occur in its dyadic expansion with asymptotic frequency $1/2$ each. Whether
$\sqrt2$ is simply normal to base $2$ is open. Isaac \cite{Isa06} claimed that
$\sqrt s$ is simply normal to base $2$ for every natural number $s$ that is not
a perfect square \cite[Thm.~1]{Isa06}. His proof has two parts. In
\cite[\S2]{Isa06} he shows that Theorem~1 follows from a property he calls
\emph{Condition~\TU} \cite[Thm.~2]{Isa06}. In \cite[\S3]{Isa06} he argues that
Condition~\TU{} holds. A later paper \cite{Isa11} restates the condition
and offers ``a simpler direct proof that Condition (TU) is true''
\cite[Thm.~1(b)]{Isa11}.

We show that Condition~\TU{} is false. Thus neither argument for it can be
correct, and the reduction in \cite[\S2]{Isa06}, while valid, cannot be used.
The simple normality of $\sqrt2$ to base $2$ remains open.

\begin{theorem}\label{thm:main}
Condition~\TU{} is false. More precisely, there are real numbers
$\nu,\nu'\in(0,1)$ with $\nu$ irrational and $\nu'-\nu\in\dy$ such that
\[
\limsup_{N\to\infty}f_N(\sqrt\nu)\ne\limsup_{N\to\infty}f_N(\sqrt{\nu'}).
\]
\end{theorem}

Here $f_N(z)$ is the relative frequency of $1$'s among the first $N$ dyadic
digits of $z$; see Section~\ref{sec:setting}. The theorem refutes \TU{} already
for the full sequence $n_i=i$. In particular, it contradicts the conclusion
$\lim_N\bigl(f_N(\sqrt\eta)-f_N(\sqrt{\eta_1})\bigr)=0$ of
\cite[Prop.~1]{Isa11}, which Isaac derives from \TU{} for irrational $\eta$
and any $\eta_1$ whose expansion differs from that of $\eta$ in finitely many
places. We give two proofs.

The first proof (Section~\ref{sec:bootstrap}) is elementary. Under \TU{}, a
local form of Isaac's own reduction shows that $f_N(y)\ge2/5$ infinitely often
whenever $y$ is irrational with $y^2$ dyadic (Proposition~\ref{prop:local}).
We apply this to the numbers $\sqrt{x_j^2+1/4}$, where $x_j$ runs through
partial sums of a very sparse dyadic series $x=\sum_j2^{-n_j}$, and choose the
positions $n_j$ so that the digits of $\sqrt{x^2+1/4}$ inherit the bound. Since
$f_N(x)\to0$, the pair $(\nu,\nu')=(x^2,x^2+1/4)$ violates \TU. The pair
exists only through the contradiction, so this proof does not exhibit a
counterexample.

The second proof (Section~\ref{sec:fourier}) does exhibit counterexamples, for
almost every point of a Cantor measure~$\mu$. By a theorem of Mosquera and
Shmerkin \cite[Thm.~3.1]{MS18}, extending a result of Kaufman \cite{Kau84}, the
image of $\mu$ under the convex map $g(\omega)=\sqrt{\omega^2/4+1/4}$ has power
Fourier decay. By Davenport, Erd\H{o}s and LeVeque \cite{DEL63}, almost every
point of $g_*\mu$ is then simply normal. The points of the support of $\mu$
have few $1$'s, so $(\omega^2/4,\omega^2/4+1/4)$ violates \TU{} for
$\mu$-almost every~$\omega$.

\begin{remark}\label{rem:lebesgue}
\TU{} does hold for Lebesgue-almost every $\nu$. By Borel's theorem and since
$\nu\mapsto\sqrt{\nu+d}$ maps null sets to null sets, for almost every $\nu$
the numbers $\sqrt{\nu+d}$ are simply normal for all $d\in\dy$ at once. So
whether \TU{} holds ``almost everywhere'' depends on the reference measure,
and the second proof shows that this cannot be repaired by passing to a
measure-theoretic form of \TU. 
\end{remark}

\section{Setting}\label{sec:setting}

For a real number $z$ and $k\ge1$ let
\[
X_k(z)=\lfloor 2^kz\rfloor \bmod 2\in\{0,1\}
\]
be the $k$-th coordinate of the dyadic expansion of $z$ (after the point), as
in \cite[(1)]{Isa06}. For dyadic rationals this picks the terminating
expansion; otherwise the expansion is unique. Put
\[
S_N(z)=X_1(z)+\dots+X_N(z),\qquad f_N(z)=\frac{S_N(z)}N ,
\]
the relative frequency of $1$'s in the first $N$ digits \cite[(3)]{Isa06}.
Two numbers $\nu,\nu'$ have the \emph{same tail} if $\nu'-\nu\in\dy$. For
$\nu,\nu'\in(0,1)$ with unique expansions this means that the two expansions
agree from some digit on, which is the definition in \cite[\S1]{Isa06}.

\begin{definition}[{Isaac \cite[\S1]{Isa06}}]\label{def:tu}
\emph{Condition~\TU} is the statement: for all $\nu,\nu'\in(0,1)$ with $\nu$
irrational and the same tail, and for every strictly increasing sequence
$(n_i)$ of natural numbers,
\[
\limsup_{i\to\infty}f_{n_i}(\sqrt\nu)=\limsup_{i\to\infty}f_{n_i}(\sqrt{\nu'}).
\]
We write $\TU_1$ for the same statement with only $n_i=i$.
\end{definition}

Isaac writes $\omega=\sqrt\nu$ and phrases \TU{} as ``$f(\omega)=h(\nu)$ is a
tail function with respect to the $U$ sequence'', the $U$ sequence being the
expansion of $\nu$. Clearly \TU{} implies $\TU_1$; both proofs below refute
$\TU_1$.

We use three elementary facts about digits.

\begin{lemma}\label{lem:shift}
Let $a,b$ be real and $s,K\in\NN$.
\begin{enumerate}
\item If $X_{k+s}(a)=X_k(b)$ for all $k\ge K$, then $f_N(a)-f_N(b)\to0$.
\item If $X_{k+s}(a)=1-X_k(b)$ for all $k\ge K$, then $f_N(a)+f_N(b)\to1$.
\end{enumerate}
\end{lemma}

\begin{proof}
In both cases $|S_N(a)-S_N(b)|$, respectively $|S_N(a)+S_N(b)-N|$, is at most
$K+2s$ for every $N$.
\end{proof}

\begin{lemma}\label{lem:carry}
Let $z,w$ be real and $m\ge1$ with $X_m(z)=0$ and
$\lfloor2^mz\rfloor\le\lfloor2^mw\rfloor\le\lfloor2^mz\rfloor+1$. Then
$\lfloor2^kw\rfloor=\lfloor2^kz\rfloor$ for all $k<m$; in particular $z$ and
$w$ have the same first $m-1$ digits. Consequently, if $X_n(z)=0$, then
$S_N(z+2^{-n})\le S_N(z)+1$ for all $N$.
\end{lemma}

\begin{proof}
Since $\lfloor2^mz\rfloor$ is even, the two values at level $m$ have the same
quotient by $2$, and one descends from $k$ to $k-1$ using
$\lfloor2^{k}t\rfloor=2\lfloor2^{k-1}t\rfloor+X_{k}(t)$. For the second
statement, $z+2^{-n}$ agrees with $z$ in the first $n-1$ digits, has digit $1$
at $n$, and agrees with $z$ after $n$.
\end{proof}

\begin{lemma}\label{lem:zero}
Let $z$ be real and $M\ge1$. There is $m>M$ with $X_m(z)=0$ and
$f_{m-1}(z)\ge f_M(z)$.
\end{lemma}

\begin{proof}
No expansion of the form $X_k=1$ for all $k>M$ occurs with our digit
convention, so there is a first $m>M$ with $X_m(z)=0$. The digits between $M$
and $m$ are all $1$, so the frequency does not decrease from $M$ to $m-1$.
\end{proof}

\section{First proof: the bootstrap}\label{sec:bootstrap}

The following is a local form of \cite[Thm.~2]{Isa06}. Isaac uses the pair
$\omega_{s_1}=1-\sqrt s/2^{2l}$, $\omega_{s_2}=(\sqrt s-1)/2^l$ of
\cite[Lemma~1]{Isa06}, whose squares have the same tail. We use an analogous
pair attached to any $y$ with dyadic square, and we need only a one-sided
bound.

\begin{proposition}\label{prop:local}
Assume $\TU_1$. Let $y$ be irrational with $1/2<y<1$ and $y^2\in\dy$. Then
$f_N(y)\ge2/5$ for infinitely many~$N$.
\end{proposition}

\begin{proof}
Put $\omega_1=y/2-1/4$ and $\omega_2=1-y/8$; both lie in $(0,1)$. Then
$\omega_2^2-\omega_1^2=15/16-15y^2/64\in\dy$, and $\omega_1^2$ is irrational
because $y$ is. Moreover $X_{k+1}(\omega_1)=X_k(y)$ for $k\ge2$, and
$X_{k+3}(\omega_2)=1-X_k(y)$ for all $k$, since $y$ is irrational. If
$f_N(y)<2/5$ for all large $N$, Lemma~\ref{lem:shift} gives
$\limsup f_N(\omega_1)\le9/20$ and $\limsup f_N(\omega_2)\ge11/20$,
contradicting $\TU_1$ for the pair $(\omega_1^2,\omega_2^2)$.
\end{proof}

Write $y(x)=\sqrt{x^2+1/4}$. Then $x^2$ and $y(x)^2$ have the same tail.

\begin{lemma}\label{lem:nonsquare}
Let $q\ge2$ and let $A<2^{q-1}$ be odd. Then $A^2+4^q$ is not a perfect
square. Consequently, if $n\ge3$ and $A<2^{n-2}$ is odd, then $y(A/2^n)$ is
irrational. It lies in $(1/2,1)$ and has dyadic square.
\end{lemma}

\begin{proof}
If $A^2+4^q=Z^2$, then $(Z-A)(Z+A)=2^{2q}$, so $Z-A=2^\alpha$, $Z+A=2^\beta$
with $\alpha<\beta$. As $A$ is odd, $\alpha=1$, whence $A=2^{2q-2}-1\ge
2^{q-1}$, a contradiction. For the second statement,
$y(A/2^n)=\sqrt{A^2+4^{n-1}}/2^n$; and $A/2^n<1/4$.
\end{proof}

\begin{lemma}\label{lem:step}
Assume $\TU_1$. Let $n\ge3$, let $A<2^{n-2}$ be odd and let $B\in\NN$. Then
there is $m\ge B$ such that $X_m(y(A/2^n))=0$ and
$f_{m-1}(y(A/2^n))\ge2/5$.
\end{lemma}

\begin{proof}
By Proposition~\ref{prop:local} and Lemma~\ref{lem:nonsquare}, there is
some $M>B$ with $f_M(y(A/2^n))\ge2/5$; now apply Lemma~\ref{lem:zero}.
\end{proof}

\noindent\textbf{The construction.} Assume $\TU_1$. Put $n_0=3$, $A_0=1$, and
for $j\ge0$ let $n_{j+1}$ be the integer $m$ given by Lemma~\ref{lem:step}
for $n=n_j$, $A=A_j$ and $B=4(j+1)n_j+4$. Put
$A_{j+1}=2^{n_{j+1}-n_j}A_j+1$. By induction, $n_j\ge3$, $A_j$ is odd and
$A_j<2^{n_j-2}$, so the construction can be continued, and
$n_{j+1}\ge4(j+1)n_j+4$, so $n_j\ge3\cdot4^j$. Let
\[
x_j=\frac{A_j}{2^{n_j}}=\sum_{i=0}^{j}2^{-n_i},\qquad
x=\sum_{i\ge0}2^{-n_i},\qquad y_j=y(x_j).
\]
Then $0<x<1/4$ and $0<x-x_j\le\frac32\,2^{-n_{j+1}}$.

\begin{lemma}\label{lem:x}
In the construction above:
\begin{enumerate}
\item $x^2$ is a Liouville number; in particular it is irrational;
\item $S_N(x)^2\le N+1$ for all $N$; in particular $f_N(x)\to0$;
\item $f_N(y(x))\ge2/5$ for infinitely many $N$.
\end{enumerate}
\end{lemma}

\begin{proof}
(1) We have $0<x^2-x_k^2\le\frac34\,2^{-n_{k+1}}<(4^{n_k})^{-k}$ because
$n_{k+1}\ge2kn_k$, and $x_k^2=A_k^2/4^{n_k}$.

(2) Given $N$, let $j$ be least with $N<n_{j+1}-1$. By Lemma~\ref{lem:carry},
applied at $m=n_{j+1}-1$, the numbers $x$ and $x_j$ have the same first $N$
digits, and $S_N(x_j)\le j+1$ by the second part of Lemma~\ref{lem:carry}. If
$j\ge1$, then $N\ge n_j-1\ge3\cdot4^j-1\ge(j+1)^2-1$.

(3) Let $m=n_{j+1}$. By the choice of $n_{j+1}$, $X_m(y_j)=0$ and
$f_{m-1}(y_j)\ge2/5$. Since $y(x)+y_j>1$,
\[
0\le y(x)-y_j\le x^2-x_j^2\le\tfrac34\,2^{-m},
\]
so $\lfloor2^my_j\rfloor\le\lfloor2^my(x)\rfloor\le\lfloor2^my_j\rfloor+1$.
By Lemma~\ref{lem:carry}, $y(x)$ and $y_j$ have the same first $m-1$ digits,
so $f_{m-1}(y(x))\ge2/5$ for every~$j$.
\end{proof}

\begin{proof}[First proof of Theorem~\ref{thm:main}]
Assume $\TU_1$ and take $x$ from the construction. Put $\nu=x^2$ and
$\nu'=x^2+1/4$. Both lie in $(0,1)$, $\nu$ is irrational by
Lemma~\ref{lem:x}(1), and $\nu'-\nu=1/4$. By Lemma~\ref{lem:x}(2),
$\limsup f_N(\sqrt\nu)=0$, while $\limsup f_N(\sqrt{\nu'})\ge2/5$ by
Lemma~\ref{lem:x}(3). This contradicts $\TU_1$.
\end{proof}

\begin{remark}
One can replace Proposition~\ref{prop:local} by \cite[Thm.~2]{Isa06}: under
\TU, $y_j=\sqrt{4^{n_j-1}+A_j^2}/2^{n_j}$ is a shift of the square root of a
non-square, hence simply normal. Proposition~\ref{prop:local} avoids this and
the passage through $\{\sqrt s\}$. The gap condition $n_{j+1}\ge4(j+1)n_j+4$
is stronger than needed for (2) and (3); it is chosen so that $x^2$ is a
Liouville number.
\end{remark}

\section{Second proof: Fourier decay}\label{sec:fourier}

Following \cite[\S2]{MS18}, for $a\in(0,1)$, translations $t=(t_1,\dots,t_m)$
and weights $p=(p_1,\dots,p_m)$ let $\mu^p_{a,t}$ be the distribution of the
random sum $\sum_{n\ge1}T_na^n$, where the $T_n$ are independent with
$\mathbb P(T_n=t_i)=p_i$. Such measures are called \emph{homogeneous
self-similar measures}. For a Borel probability measure $\mu$ on $\RR$ and a
map $F$, write $F\mu$ for the image measure and
$\widehat\mu(u)=\int e^{-2\pi iux}\,d\mu(x)$. We use two theorems from the
literature.

\begin{theorem}[{Mosquera--Shmerkin \cite[Thm.~3.1]{MS18}}]\label{thm:ms}
Let $F\in C^2(\RR)$ with $F''>0$, and let $\mu=\mu^p_{a,t}$ be a homogeneous
self-similar measure with all $p_i>0$ and without atoms. Then there are
$\sigma>0$ and $C>0$ such that
\[
|\widehat{F\mu}(u)|\le C|u|^{-\sigma}\qquad(|u|\ge1).
\]
\end{theorem}

\begin{theorem}[{Davenport--Erd\H{o}s--LeVeque \cite{DEL63}; see \cite[\S1]{MS18}}]\label{thm:del}
Let $\mu$ be a Borel probability measure on $\RR$ with
$|\widehat\mu(u)|\le C|u|^{-\sigma}$ for $|u|\ge1$, where $\sigma>0$. Then
$\mu$-almost every $z$ is simply normal to base $2$, that is,
$f_N(z)\to1/2$.
\end{theorem}

\begin{remark}
In \cite[Thm.~3.1]{MS18} it is only assumed that $\mu$ is not a single
atom; we assume that $\mu$ has no atoms, which is all we need. The theorem of \cite{DEL63} is stated
for Lebesgue measure on an interval: if $\sum_NN^{-1}I(N)<\infty$ for every
$m\ne0$, where $I(N)$ is the mean square of the Weyl sum
$N^{-1}\sum_{n\le N}e(ms_n(x))$, then $s_n(x)$ is uniformly distributed
modulo~$1$ for almost all~$x$. Its proof \cite[\S2]{DEL63} uses only that the
measure is finite, so it holds for any Borel probability measure~$\mu$. For
$s_n(x)=2^nx$ one has
\[
I(N)=N^{-2}\sum_{n,n'\le N}\widehat\mu\bigl(m(2^n-2^{n'})\bigr)\ll N^{-1}
\]
under power decay, so $\mu$-almost every $z$ is normal to base~$2$, as noted in
\cite[\S1]{MS18}. We record only simple normality to base~$2$.
\end{remark}

Let $\mu=\mu^p_{a,t}$ with $a=1/8$, $t=(0,1)$ and $p=(1/2,1/2)$: the natural
measure on the Cantor set of ratio $1/8$. It is the image of the product of
fair coins on $\{0,1\}^\NN$ under
$\pi(\iota)=\sum_{k\ge0}\iota_k8^{-(k+1)}$.

\begin{lemma}\label{lem:cantor}
\begin{enumerate}
\item The map $\pi$ is injective and $\mu$ has no atoms.
\item For every $\iota$, $X_k(\pi(\iota))=0$ unless $3\mid k$. Hence
$S_N(\pi(\iota))\le\lfloor N/3\rfloor$ for all $N$.
\end{enumerate}
\end{lemma}

\begin{proof}
For every $m$, $8^m\pi(\iota)=\sum_{k<m}\iota_k8^{m-1-k}+R$ with
$0\le R\le1/7$. Hence $\lfloor8^m\pi(\iota)\rfloor$ determines
$\iota_0,\dots,\iota_{m-1}$, and $\lfloor2^{3m+r}\pi(\iota)\rfloor$ is even
for $r=1,2$. Each point of $\{0,1\}^\NN$ has product measure at most $2^{-n}$
for every~$n$.
\end{proof}

\begin{proof}[Second proof of Theorem~\ref{thm:main}]
Let $g(\omega)=\sqrt{\omega^2/4+1/4}$. Then $g$ is smooth with
$g''=1/(16g^3)>0$. By Theorems~\ref{thm:ms} and~\ref{thm:del} and
Lemma~\ref{lem:cantor}(1), $g(\omega)$ is simply normal to base $2$ for
$\mu$-almost every $\omega$. Since $\mu$ has no atoms, also $\omega^2$ is
irrational for $\mu$-almost every $\omega$: the exceptional set is contained
in the countable set $\{\pm\sqrt q:q\in\QQ\}$. Fix $\omega$ with both
properties. Then $0<\omega\le1/7$. Put $\nu=\omega^2/4$ and
$\nu'=\nu+1/4=g(\omega)^2$. Both lie in $(0,1)$, $\nu$ is irrational, and
$\nu'-\nu\in\dy$. Since $X_{k+1}(\omega/2)=X_k(\omega)$,
Lemma~\ref{lem:cantor}(2) gives $f_N(\omega/2)\le1/3+2/N\le2/5$ for $N\ge30$.
Hence
\[
\limsup f_N(\sqrt\nu)\le\tfrac25<\tfrac12=\lim f_N(\sqrt{\nu'}).\qedhere
\]
\end{proof}

\begin{remark}
(i) The measure $\mu$ is an affine image of the Bernoulli convolution
$\nu_{1/8}$, the distribution of $\sum_{n\ge1}\pm8^{-n}$. According to
\cite[\S1]{MS18}, Kaufman \cite{Kau84} proved $\dim_F(F\nu_a)>0$ for all
$a\in(0,1/2)$ and all $C^2$ diffeomorphisms $F$ with $F''>0$, so for this
measure Theorem~\ref{thm:ms} is already due to him; Theorem~\ref{thm:ms}
extends and quantifies his argument.

(ii) The proof shows more: \TU{} fails at $(\omega^2/4,\omega^2/4+1/4)$ for
$\mu$-almost every $\omega$, a set of Hausdorff dimension $1/3$. The same
argument works for the Bernoulli measure of parameter $p\ne1/2$ on the dyadic
digits. That measure gives a set of dimension close to $1$ for $p$ close to
$1/2$. It requires the strong law of large numbers in place of
Lemma~\ref{lem:cantor}(2).
\end{remark}

\appendix
\section{Formal verification}
\sloppy

All definitions, lemmas, propositions and theorems of this note are proved
in Lean~4 with Mathlib in the file \texttt{BB47/Isaac.lean}; the remarks are
not part of the formal verification. Theorems~\ref{thm:ms}
and~\ref{thm:del} are cited results. They are recorded as axioms in
\texttt{CITED/MosqueraShmerkin.lean} and
\texttt{CITED/DavenportErdosLeVeque.lean}. The first proof uses only the
standard axioms of Lean; the second uses, in addition, exactly these two. The
files are available in the Github repository
\texttt{https://github.com/rwst/lean-code}.

\medskip
{\small
\noindent
\begin{tabular}{@{}l>{\raggedright\arraybackslash}p{0.6\textwidth}@{}}
\toprule
Statement & Lean declaration (namespace \texttt{BB47.Isaac}) \\
\midrule
Definition~\ref{def:tu} & \texttt{TU}, \texttt{TUfull}, \texttt{SameTail}, \texttt{freq} \\
Lemma~\ref{lem:shift} & \texttt{tendsto\_freq\_sub}, \texttt{tendsto\_freq\_add} \\
Lemma~\ref{lem:carry} & \texttt{floor\_agree}, \texttt{cnt\_agree}, \texttt{cnt\_add\_pow} \\
Lemma~\ref{lem:zero} & \texttt{exists\_zero\_freq} \\
Proposition~\ref{prop:local} & \texttt{frequently\_freq\_ge} \\
Lemma~\ref{lem:nonsquare} & \texttt{not\_isSquare\_add\_four\_pow}, \texttt{ysq\_dyadic\_irrational} \\
Lemma~\ref{lem:step} & \texttt{bootstrap\_step} \\
Lemma~\ref{lem:x} & \texttt{liouville\_sq\_xlim}, \texttt{tendsto\_freq\_xlim}, \texttt{frequently\_freq\_ylim} \\
Theorem~\ref{thm:main}, first proof & \texttt{not\_TUfull\_bootstrap}, \texttt{not\_TU\_bootstrap} \\
Theorem~\ref{thm:ms} (axiom) & \texttt{MS.fourierDecay\_smoothImage} \\
Theorem~\ref{thm:del} (axiom) & \texttt{DEL.ae\_simplyNormal\_of\_fourierDecay} \\
Lemma~\ref{lem:cantor} & \texttt{cantor8\_injective}, \texttt{mu8\_singleton}, \texttt{cnt\_cantor8\_le} \\
Theorem~\ref{thm:main}, second proof & \texttt{not\_TUfull\_fourier}, \texttt{not\_TU\_fourier} \\
\bottomrule
\end{tabular}
}



\begin{thebibliography}{DEL63}

\bibitem[DEL63]{DEL63}
H.~Davenport, P.~Erd\H{o}s and W.~J.~LeVeque, \emph{On Weyl's criterion for
uniform distribution}, Michigan Math.\ J.\ \textbf{10} (1963), 311--314.

\bibitem[Isa06]{Isa06}
R.~Isaac, \emph{On the simple normality to base 2 of $\sqrt s$, for $s$ not a
perfect square}, preprint (2006), arXiv:math/0512404v2.

\bibitem[Isa11]{Isa11}
R.~Isaac, \emph{On a sufficient condition that the square root of $s$ is
simply normal to base 2, for $s$ not a perfect square}, preprint (2011),
arXiv:1104.1616.

\bibitem[Kau84]{Kau84}
R.~Kaufman, \emph{On Bernoulli convolutions}, in: Conference in Modern
Analysis and Probability (New Haven, 1982), Contemp.\ Math.\ \textbf{26},
Amer.\ Math.\ Soc., Providence, RI, 1984, 217--222.

\bibitem[MS18]{MS18}
C.~A.~Mosquera and P.~S.~Shmerkin, \emph{Self-similar measures: asymptotic
bounds for the dimension and Fourier decay of smooth images}, Ann.\ Acad.\
Sci.\ Fenn.\ Math.\ \textbf{43} (2018), 823--834.

\end{thebibliography}
\end{document}